\documentclass[a4paper,12pt]{amsart}
\usepackage{amssymb,amsmath,amsfonts,amsthm}
\usepackage[all]{xy}
\usepackage{enumerate}
\usepackage{mathrsfs}
\usepackage{graphicx}
\usepackage{thm-restate}
\usepackage{multirow}

\def\aa{{\bf a}} 

\newcommand{\N}{\mathbb{N}}
\newcommand{\bb}{\mathfrak{b}}
\newcommand{\point}{\mathfrak{p}}
\newcommand{\m}{\mathfrak{m}}
\newcommand{\MS}{\mathfrak{S}}
\newcommand{\Ker}{{\operatorname{Ker}}}
\newcommand{\ord}{{\operatorname{ord}}}
\newcommand{\Ap}{{\operatorname{Ap}}}
\newcommand{\gr}{{\operatorname{gr}}}

\newcommand{\bmt}[1]{\begin{bmatrix}#1\end{bmatrix}}

\newtheorem{pro}{Proposition}[section]
\newtheorem{Lem}[pro]{Lemma}
\newtheorem{Cor}[pro]{Corollary} 
\newtheorem{Theo}[pro]{Theorem}

\newtheorem*{Theoetoile}{Theorem} 
\theoremstyle{definition}
\newtheorem{Defi}[pro]{Definition}

\newtheorem{Data}[pro]{Setting}

\newtheorem{Rem}[pro]{Remark}
 
\numberwithin{equation}{section}
\usepackage{hyperref}
\hypersetup{
	colorlinks=false,       
	linkcolor=red,          
	citecolor=blue,        
	filecolor=magenta,      
	urlcolor=cyan           
}

\title{{\small   The weak Lefschetz property of Gorenstein algebras of codimension three associated to the Ap\'ery sets}}
\date{\today}

\author{Rosa M. Mir\'o-Roig}
\address{Universitat de Barcelona, Departament de Matem\`atiques i Inform\`atica, Gran Via de les Corts Catalanes 585, 08007 Barcelona, Spain.}
\email{miro@ub.edu}

\author{Quang Hoa Tran}
\address{University of Education, Hue University,  34 Le Loi St., Hue City, Vietnam.}
\email{tranquanghoa@hueuni.edu.vn}

\begin{document}
\maketitle
\begin{abstract}
It has been conjectured that {\it all} graded Artinian Gorenstein algebras of codimension three have the weak Lefschetz property  over a field of characteristic zero. In this paper, we study the weak Lefschetz property of  associated graded  algebras $A$ of the Ap\'ery set of $M$-pure symmetric numerical semigroups generated  by four natural numbers. In 2010, Bryant proved that these algebras are graded Artinian Gorenstein algebras of codimension three. In a recent article, Guerrieri showed that if $A$ is not a complete intersection, then $A$ is of form $A=R/I$ with $R=K[x,y,z]$ and
\begin{align*} 
I=(x^a, y^b-x^{b-\gamma} z^\gamma, z^c, x^{a-b+\gamma}y^{b-\beta}, y^{b-\beta}z^{c-\gamma}),
\end{align*}
where $ 1\leq \beta\leq b-1,\; \max\{1, b-a+1 \}\leq \gamma\leq \min \{b-1,c-1\}$ and $a\geq c\geq 2$. We prove that $A$ has the  weak Lefschetz property  in the following cases:
\begin{enumerate}[$\bullet$]
\item $ \max\{1,b-a+c-1\}\leq \beta\leq b-1$ and $\gamma\geq \lfloor\frac{\beta-a+b+c-2}{2}\rfloor$;
\item $ a\leq 2b-c$ and $| a-b| +c-1\leq \beta\leq b-1$;
\item one of $a,b,c$ is at most five. 
\end{enumerate} 
\texttt{MSC2010}: primary 13E10, 13H10; secondary 13A30, 13C40.
\end{abstract}

\section{Introduction}
The weak Lefschetz property (WLP for short) for an Artinian graded algebra $A$ over a field $K$ simply says that the multiplication by a general linear form  $\times L: [A]_{i}\longrightarrow [A]_{i+1}$ has maximal rank in every degree $i$. At first glance this might seem to be a simple problem of linear algebra. However, determining which graded Artinian $K$-algebras have the WLP is notoriously difficult. Many authors have studied the problem from many different points of view, applying tools from representation theory, topology, vector bundle theory, plane partitions, splines, differential geometry, commutative algebra, among others (see for instance \cite{BK2007, HSS2011,   MMO2013, MMR03,MMR17, MMN2011, MMN2012,MN2013,  MiroRoig2016,  MRT19, Stanley80, Watanabe1987}). Even the characteristic of $K$ plays an interesting role; see for example \cite{BK2011,CookII2012,CN2011,LF2010,MMN2011}. 

One of the most interesting open problems in this field is whether \textit{all} graded Artinian Gorenstein algebras of codimension three have the WLP in characteristic zero. In the special case of codimension three complete intersections, a positive answer was obtained in characteristic zero in \cite{HMNW2003} using the Grauert-M\"{u}lich theorem. For positive characteristic, on the other hand, only the case of monomial complete intersections has been studied (see \cite{BK2011,CookII2012,CN2011}), applying many different approaches from combinatorics.

For the case of Gorenstein algebras of codimension three that are not necessarily complete intersections, it was known that for each possible Hilbert function an example exists having the WLP \cite{Harima1995}. Some partial results are given in \cite{MZ2008} to show that for certain Hilbert functions, all such Gorenstion algebras have the WLP. It was shown in \cite{BMMNZ14} that  all codimension three Artinian Gorenstein algebras of socle degree at most 6 have the WLP in characteristic zero.  But the general case remains completely open.

In this work, we consider a class of graded Artinian Gorenstein algebras of codimension three built up starting from the  Ap\'ery set of a numerical semigroup generated by four natural numbers. Our goal is to study whether these algebras have the WLP.  More precisely,  we consider a numerical semigroup $P$ generated by  $\{a_1,a_2,a_3,a_4\}\subset \N^4$ with $\gcd(a_1,a_2,a_3,a_4)=1$.  The \textit{Ap\'ery set $\Ap(P)$ of $P$} with respect to the minimal generator of the semigroup is defined as follows
$$\Ap(P):=\{a\in P\mid a-a_1\notin P\}=\{0=\omega_1<\omega_2<\cdots<\omega_{a_1}\}.$$
Notice that $\Ap(P)$ is a finite set and $\# \Ap(P)=a_1.$  Recall that a numerical semigroup $P$ is said to be \textit{$M$-pure symmetric} if for each $i=1,\ldots,a_1$, $\omega_i+\omega_{a_1-i+1}=\omega_{a_1}$ and $\ord(\omega_i)+\ord(\omega_{a_1-i+1})=\ord(\omega_{a_1})$, where 
$$ \ord(a):=\max \big\{ \sum_{i=1}^{4} \lambda_i \mid  a= \sum_{i=1}^{4} \lambda_ia_i  \big\}$$
is the \textit{order} of $a\in P$. Therefore the Ap\'ery set of a $M$-pure symmetric semigroup has the structure of a symmetric lattice.

Let $K$ be a field of characteristic zero and consider the homomorphism 
$$\Phi: S:=K[x_1,\ldots,x_4]\longrightarrow K[P]:=K[t^{a_1},\ldots,t^{a_4}],$$
which sends $x_i\longmapsto t^{a_i}.$ Then $K[P]\cong S/\Ker(\Phi)$ is a one dimensional ring associated to $P$. Now set $\overline{S} =S/(x_1).$ Then there is one to one correspondence between the elements of $\Ap(P)$ and the generators of $\overline{S}$ as a $K$-vector space. Let $\overline{\m}$ be the maximal homogeneous ideal of $\overline{S}$, define the \textit{associated graded algebra} of the Ap\'ery set of $P$
$$A=\gr_{\overline{\m}}(\overline{S}):=\bigoplus_{i\geq0}\frac{\overline{\m}^i}{\overline{\m}^{i+1}}.$$
It follows that $A$ is a standard graded Artinian $K$-algebra of codimension three. In the work \cite{Bryant2010}, Bryant proved that $A$ is Gorenstein if and only if $P$ is $M$-pure symmetric. In a recent paper \cite{Gu2018}, Guerrieri showed that if $A$ is an Artinian Gorenstein algebra that is not a complete intersection, then $A$ is of form $A=R/I$ with $R=K[x,y,z]$ and  
\begin{align} \label{idealGorenstein}
I=(x^a, y^b-x^{b-\gamma} z^\gamma, z^c, x^{a-b+\gamma}y^{b-\beta}, y^{b-\beta}z^{c-\gamma}),
\end{align}
where $ 1\leq \beta\leq b-1,\; \max\{1, b-a+1 \}\leq \gamma\leq \min \{b-1,c-1\}$ and $a\geq c\geq 2$. The integers $a,b,c,\beta$ and $\gamma$ are determined by the structure of $\Ap(P)$, see \cite[Section 5]{Gu2018}.

Our goal is to study the WLP of $A$. This is a continuation of our recent article \cite{MRT2019}. Our main results are the following (see Theorems~\ref{Theorem3.7}, ~\ref{Theorem3.8},~\ref{Theorem4.3} and~\ref{Theorem5.3}; Corollaries~\ref{Corollary3.6}~--~\ref{Corollary3.13}).
\begin{Theoetoile}
Consider the ideal $I$ as in \eqref{idealGorenstein}. Then $R/I$ has the WLP  if one of the following conditions is satisfied:
\begin{enumerate}
\item [\rm (i)] $ \max\{1,b-a+c-1\}\leq \beta\leq b-1$ and $\gamma\geq \lfloor\frac{\beta-a+b+c-2}{2}\rfloor$;
\item [\rm (ii)] $a\leq 2b-c$ and $| a-b| +c-1\leq \beta\leq b-1$;
\item [\rm (iii)] $a\geq 2b+c-6$;
\item [\rm (iv)]$a\geq b+c-2$ and $1\leq \beta\leq a-b-c+5$;
\item [\rm (v)] one of $a,b,c$ is at most five. 
\end{enumerate}
\end{Theoetoile}
Notice that the condition (iv) is a special case of (i) and the condition (iii) is a consequence of (i) and (v). 

The paper is organized as follows. 
In Section~\ref{Section2}, after recalling  some basic results  on the WLP of Artinian Gorenstein algebras, we prove a technical lemma that decides the singularity of a class of matrices (Lemma~\ref{Lemmadeterminant}). 
From Section~\ref{Section3}, we study the WLP of the ideal as in Setting \ref{idealGorenstein}. The heart of the paper is Section~\ref{Section3}, we prove the two main Theorems~\ref{Theorem3.7} and~\ref{Theorem3.8} where it is shown that $R/I$ has the WLP  for most of the choices of the parameters $a,b,c,\beta$ and $\gamma$. By applying these theorems, we recover our main results in  \cite{MRT2019} (see Corollaries~\ref{Corollary3.6} and~\ref{Corolarry3.9}). 
In the last two  sections, we show that $R/I$ has the WLP if one of $a,b$ and $c$ is equal to  four  or five in Theorem~\ref{Theorem4.3} or Theorem~\ref{Theorem5.3}, respectively.

\section{Background and preparatory results}\label{Section2}
In this section, we fix notations and we recall known facts needed later in this work. We fix $K$ a field of characteristic zero and $R=K[x_1,\ldots,x_n]$. Let
$$A=R/I=\bigoplus_{i=0}^D [A]_i$$
be a graded Artinian algebra. 

\begin{Defi}
For any graded Artinian  algebra $A=R/I=\bigoplus_{i=0}^D [A]_i$, the \textit{Hilbert function} of $A$ is a function
$$h_A: \N\longrightarrow \N$$
defined by $h_A(t)=\dim_K [A]_t$. As $A$ is Artinian, its Hilbert function is equal to its \textit{$h$-vector} that one can express as a sequence
$$\underline{h}_A=(1=h_0,h_1,h_2,h_3,\ldots, h_D),$$
with $h_i=h_A(i)>0$ and $D$ is the last index with this property. The integer $D$ is called the \textit{socle degree} of $A$. The $h$-vector  $\underline{h}_A$ is said to be \textit{symmetric} if $h_{D-i}=h_i$ for every $i=0,1,\ldots,\lfloor\frac{D}{2}\rfloor.$ 
\end{Defi}

\begin{Defi} \cite[Proposition~2.1]{MW2009}
A graded Artinian algebra $A$ as above is Gorenstein if and only if $h_D=1$ and the multiplication map 
$$[A]_i\times [A]_{D-i}\longrightarrow [A]_D\cong K$$ is a perfect pairing for all $i=0,1,\ldots,\lfloor\frac{D}{2}\rfloor.$
\end{Defi}
It follows that the $h$-vector of a graded Artinian Gorenstein is symmetric.

\begin{Defi}
A graded Artinian $K$-algebra $A$ has the \textit{weak Lefschetz property}, briefly WLP, if there exists $L\in [A]_1$ such that the map $\times L: [A]_i\longrightarrow [A]_{i+1}$ has maximal rank for each $i$. We also say that a homogeneous ideal $I$ has the WLP if $R/I$ has the WLP.
\end{Defi}
We can determine the WLP by considering the rank of the multiplication map by  a general linear form in every degree. However, for a graded Artinian Gorenstein $K$-algebra, the WLP is determined by considering only the multiplication map in one degree (see \cite[Proposition 2.1]{MMN2011}). Moreover, we have the following.
\begin{pro}\cite[Theorem 2.1]{Gu2018}\label{Proposition2.7}
Assume that $G=\bigoplus_{i=0}^D [G]_i$ is a standard graded Artinian Gorenstein $K$-algebra with socle degree $D$ that has the WLP. If $\ell\in [G]_1$ is a linear element, then the quotient ring
$$A=\frac{G}{(0\colon_G \ell)}$$
is also a standard graded Artinian Gorenstein $K$-algebra. Furthermore, if $G$ and $A$ have the same codimension and the socle degree $D$  is odd, then $A$ has also the WLP.
\end{pro}
We close this section by giving the following technical lemma which is useful to decide the singularity of matrices in the next sections.
\begin{Lem}\label{Lemmadeterminant}
Let $n\geq 2$ be an integer and  denote by $\MS_n$ the set of all square matrices of order $n$ of form
\begin{align*}
A=\bmt{1& -a_{12}&\cdots&-a_{1n-1}&-a_{1n}\\ 0&1& \cdots&-a_{2n-1}&-a_{2n}\\\vdots&\vdots&\ddots&\vdots&\vdots\\
0&0&\cdots&1&-a_{n-1n}\\
a_1&a_2&\cdots&a_{n-1}&a_n},
\end{align*}
where $a_{ij}\geq 0$ for all $1\leq i<j\leq n,\; a_i\geq 0$ for all $i=1,2,\ldots,n-1$ and $a_n>0.$  Then, for any integer $n\geq 2$ and $A\in \MS_n$, $\det(A)>0.$
\end{Lem}
\begin{proof}
First we define a sequence of non-negative integers $(\alpha_j)_{j=1,\ldots,n}$ as follows: $\alpha_1=a_1\geq 0$ and for any $2\leq j\leq n$
$$\alpha_j= a_j+\sum_{i=1}^{j-1} \alpha_i a_{ij}.$$
Clearly, $\alpha_i\geq a_i\geq 0$ for all $1\leq i\leq n$ and in particular  $\alpha_n\geq a_n>0.$ Now  we denote by $d_i$ the $i$-th row of the matrix $A$. Let $\bar{A}$ be the matrix obtained from $A$ by replacing $d_n$ by $d_n -\sum_{j=1}^{n-1}\alpha_j d_j$. Then  
\begin{align*}
\bar{A}=\bmt{1& -a_{12}&\cdots&-a_{1n-1}&-a_{1n}\\ 0&1&\cdots&-a_{2n-1}&-a_{2n}\\
	\vdots&\vdots&\ddots&\vdots&\vdots\\
0&0&\cdots&1&-a_{n-1n}\\
	0&0&\cdots&0&\alpha_n}.
\end{align*}
By the basic property of the determinant, we get
$\det(A)=\det(\bar{A})=\alpha_n>0.$
\end{proof}

\section{The WLP for a class of Artinian Gorenstein algebras of codimension three}\label{Section3}
From now on, we study the WLP of  the ideal  in the following setting.
\begin{Data}\label{IdealGorenstein}
Let $R=K[x,y,z]$ be the standard graded polynomial ring over a field $K$ of characteristic zero and consider the ideal
\begin{align*}
I=(x^a, y^b-x^{b-\gamma} z^\gamma, z^c, x^{a-b+\gamma}y^{b-\beta}, y^{b-\beta}z^{c-\gamma})\subset R,
\end{align*}
where  $1\leq \beta \leq b-1,\ \max\{1, b-a+1\}\leq \gamma\leq \min\{b-1,c-1\}$ and $a\geq c\geq 2.$ 
\end{Data}
It is clear that $a\geq b-c+2$. Firstly, we need the following result.

\begin{pro} \cite[Proposition 3.1]{MRT2019}\label{Proposition3.1}
 Let  $I_\beta:=I$ be as in Setting~\ref{IdealGorenstein}.  Set $A_\beta=R/I_\beta$ and $G= R/\aa$, where $\aa:= (x^a, y^b-x^{b-\gamma} z^{\gamma}, z^c)$. Then one has:
\begin{enumerate}
\item[\rm (i)] $I_\beta=\aa \colon_R y^\beta.$ Therefore, $A_\beta$ is an Artinian Gorenstein algebra of codimension three and its socle degree is $D=a+b+c-\beta-3.$
\item[\rm (ii)] $G$ is an Artinian complete intersection of codimension three with socle degree $a+b+c-3$ and hence it has the WLP. For all $2\leq \beta\leq b-1$
$$ A_1 =\frac{G}{(0\colon_{G} y)}\quad \text{and}\quad A_\beta =\frac{A_{\beta -1}}{(0\colon_{A_{\beta-1}} y)}.$$
\end{enumerate}
\end{pro}
\begin{Rem} \label{Remark3.3}
Let  $I_\beta:=I$ be as in Setting~\ref{IdealGorenstein}. Set $A_\beta=R/I_\beta$ and $k=\lfloor\frac{a+b+c-\beta-3}{2}\rfloor$.
\begin{enumerate}
\item[(i)] Let $u$ and $v$ be two integers such that $1\leq u\leq v\leq b-1$. To prove $A_\beta$ has the WLP for any $u\leq \beta\leq v$, by Propositions~\ref{Proposition2.7} and~\ref{Proposition3.1}, it is enough to show that $A_\beta$ has the WLP with integers $\beta$ satisfying $a+b+c-\beta$ is even and $\beta=u$ whenever $a+b+c-u$ is odd.
\item[(ii)] Set $L=x-y-z$. By  \cite[Proposition 2.1]{MMN2011}), $A_\beta$ has the WLP if and only if   $[A_\beta/LA_\beta]_{k+1}=0$. Set
$$J=((y+z)^{a}, y^b-(y+z)^{b-\gamma} z^\gamma, z^{c}, (y+z)^{a-b+\gamma}y^{b-\beta}, y^{b-\beta}z^{c-\gamma}).$$
 Then $A_\beta/LA_\beta\cong K[y,z]/J$, and hence  $A_\beta$ has the WLP if and only if   $[K[y,z]/J]_{k+1}=0$; or equivalently, $y^{k+1-i}z^i\in J$ for all $0\leq i\leq k+1.$
\item[(iii)] Since $\frac{a+b+c-\beta-3}{2}+1-c =\frac{a-c+b-1-\beta}{2}\geq 0$,  $k+1\geq c$. It follows that $y^{k+1-i}z^i\in J$ for any $c\leq i\leq k+1$ as $z^c\in J.$
\end{enumerate}
\end{Rem}

We now prove our first main result in this section.

\begin{Theo}\label{Theorem3.7}
Let  $I$ be as in Setting~\ref{IdealGorenstein}.  Then  $R/I$ has the WLP, whenever $\max\{1,b+c-a-1\}\leq \beta \leq b-1$ and $\gamma\geq \lfloor\frac{\beta-a+b+c-2}{2}\rfloor.$
\end{Theo}
Note that if  $\beta \leq a-b-c+5$, then the assumption $\gamma\geq \lfloor\frac{\beta-a+b+c-2}{2}\rfloor$ is automatically satisfied.
\begin{proof}
Set $I_\beta:=I$ and $A_\beta =R/I_\beta$. By Remark~\ref{Remark3.3}(i), it is enough to show that $A_\beta$ has the WLP,  provided $\max\{1,b+c-a\}\leq \beta \leq b-1,\, a+b+c-\beta$ is even and $\gamma\geq \frac{\beta-a+b+c-2}{2}$; or $\beta=b+c-a-1\geq 1$ and $\gamma\geq \beta-1$.  Under these conditions, set $k=\lfloor\frac{a+b+c-\beta-3}{2}\rfloor$ and  $L=x-y-z$.  By (ii) and (iii) of Remark~\ref{Remark3.3}, it suffices to show  that  
$y^{k+1-i}z^i  \in J$  for any $0\leq i\leq c-1,$ where 
$$J=((y+z)^{a}, y^b-(y+z)^{b-\gamma} z^\gamma, z^{c}, (y+z)^{a-b+\gamma}y^{b-\beta}, y^{b-\beta}z^{c-\gamma}).$$
To do it, we consider the following two cases.

\noindent{\bf \underline{Case 1: $\max\{1,b+c-a\}\leq \beta\leq b-1, a+b+c-\beta$ is even and $2\gamma\geq \beta-a+b+c-2$.}} In this case, one has $k=\frac{a+b+c-\beta-4}{2}$. For each $c-\gamma\leq i\leq c-1$, one has 
$$k+1-b+\beta-i\geq k+2-b-c+\beta =\frac{a-b-c+\beta}{2}\geq 0,$$
as $\beta \geq b+c-a$. Hence   
\begin{align*}
y^{k+1-i} z^{i} =y^{b-\beta}z^{c-\gamma} y^{k+1-b+\beta-i}z^{i-c+\gamma} \in J
\end{align*}
since $y^{b-\beta}z^{c-\gamma}\in J$. 

It follows that $y^{k+1-i}z^i\in J$ for any $c-\gamma\leq i\leq k+1.$  It remains to show the following  assertion.

\noindent\textsc{ \underline{Claim:}} $y^{k+1-i}z^i\in J$ for all $0\leq i\leq c-\gamma-1.$

We prove this claim  by dividing into the following two subcases.

\noindent{\bf \underline{Subcase 1:  $2\gamma\geq \beta-a+b+c$.}} In this case, one has
\begin{align*}
k+2-b-c+\gamma=\frac{a-b-c-\beta+2\gamma}{2}\geq 0.
\end{align*}
Hence $k+2-b\geq c-\gamma.$ It follows that $y^{k+1-i}z^i\in J$ for any $k+2-b\leq i\leq k+1.$ Clearly, $k+1\geq b$ as $k+2-b\geq c-\gamma\geq 1.$  Furthermore, as $y^b-(y+z)^{b-\gamma}z^\gamma\in J$ and by the Newton binomial expansion formula
$$y^b-(y+z)^{b-\gamma}z^\gamma=y^b-\sum_{j=0}^{b-\gamma}\binom{b-\gamma}{j}y^{b-\gamma-j}z^{\gamma+j},$$
we get 
$$ y^{k+1-i}z^i-\sum_{j=0}^{b-\gamma}\binom{b-\gamma}{j}y^{k+1-\gamma-i-j}z^{\gamma+i+j} = (y^b-(y+z)^{b-\gamma}z^\gamma)y^{k+1-b-i}z^i\in J$$
for any $0\leq i\leq k+1-b$. This gives a system of  $(k+2-b)$ equations in $(k+2-b)$ indeterminates $y^{k+1}, y^{k}z,\ldots, y^{b}z^{k+1-b}$ as follows
\begin{align*}
M X=U,
\end{align*}
where $X^t =\bmt{y^{k+1} & y^{k}z&\cdots& y^{b}z^{k+1-b}},\; U^t=\bmt{u_1&u_2&\cdots &u_{k+2-b}}$ with  $u_i\in J$ and
the coefficient matrix $M$ is an upper triangular matrix such that the entries on the main diagonal are all 1. Thus, by Cramer's rule,  $y^{k+1}, y^{k}z,\ldots, y^{b}z^{k+1-b}\in J.$

\noindent{\bf \underline{Subcase 2: $2\gamma= \beta-a+b+c-2$.}} In this case, we have
\begin{align*}
k+3-b-c+\gamma=\frac{a-b-c-\beta+2+2\gamma}{2}= 0.
\end{align*}
Hence $k+3-b= c-\gamma.$ It follows that $y^{k+1-i}z^i\in J$ for any $k+3-b\leq i\leq k+1.$ Furthermore, as $y^b-(y+z)^{b-\gamma}z^\gamma\in J$ and by the Newton binomial expansion formula
$$y^b-(y+z)^{b-\gamma}z^\gamma=y^b-\sum_{j=0}^{b-\gamma}\binom{b-\gamma}{j}y^{b-\gamma-j}z^{\gamma+j},$$
we get 
$$ y^{k+1-i}z^i-\sum_{j=0}^{b-\gamma}\binom{b-\gamma}{j}y^{k+1-\gamma-i-j}z^{\gamma+i+j} = (y^b-(y+z)^{b-\gamma}z^\gamma)y^{k+1-b-i}z^i\in J$$
for any $0\leq i\leq k+1-b$. On the other hand, $k+1=a-\beta+\gamma$ and  $(y+z)^{a-b+\gamma}y^{b-\beta}\in J$. Again by the Newton binomial expansion formula
\begin{align*}
(y+z)^{a-b+\gamma}y^{b-\beta}=\sum_{j=0}^{a-b+\gamma} \binom{a-b+\gamma}{j}y^{a-\beta+\gamma-j}z^j\in J.
\end{align*}
It follows that we can establish  a system of  $(k+3-b)$ equations in $(k+3-b)$ indeterminates $y^{k+1}, y^{k}z,\ldots, y^{b-1}z^{k+2-b}$ as follows
\begin{align*}
M X=U,
\end{align*}
where $X^t =\bmt{y^{k+1} & y^{k}z&\cdots& y^{b-1}z^{k+2-b}},\; U^t=\bmt{u_1&u_2&\cdots &u_{k+3-b}}$ with  $u_i\in J$ and
the coefficient matrix 
\begin{align*}
M=\bmt{1&0&\cdots&-1&-\binom{b-\gamma}{1}&\cdots&-\binom{b-\gamma}{b-\gamma}&\cdots&0&0\\
	0&1&\cdots&0&-1&\cdots&-\binom{b-\gamma}{b-\gamma-1}&\cdots&0&0\\
	\vdots&\vdots&\cdots&\vdots&\vdots&\cdots&\vdots&\cdots&\vdots&\vdots\\
	0&0&\cdots&0&0&\cdots&0&\cdots&1&0\\
	1&\binom{a-b+\gamma}{1}&\cdots&\binom{a-b+\gamma}{b}&\binom{a-b+\gamma}{\gamma+1}&\cdots&\binom{a-b+\gamma}{b}&\cdots
	&\binom{a-b+\gamma}{k+1-b}&\binom{a-b+\gamma}{k+2-b} }.
\end{align*} 
It follows that
$M\in  \MS_{k+3-b}$, hence $\det(M)>0$ by Lemma~\ref{Lemmadeterminant}. Thus $y^{k+1-i}z^i\in J$ for any $0\leq i\leq k+2-b$ by Cramer's rule.


\noindent{\bf \underline{Case 2: $\beta=b-a+c-1\geq 1$ and $\gamma\geq \beta-1$.}} In this case, one has $k=a-1$.  We must prove that  $y^{a-i}z^i\in J$ for all $0\leq i\leq c-1.$ As $y^{b-\beta}z^{c-\gamma}=y^{a-c+1}z^{c-\gamma}\in J$, we get
$$y^{a-i}z^i =y^{a-c+1}z^{c-\gamma} y^{c-1-i}z^{i-c+\gamma}\in J\quad \text{for any}\; c-\gamma\leq i\leq c-1.$$
As $\gamma\geq \beta-1$, we get $a-b+2\geq c-\gamma$ and hence $a\geq b-1$ and $y^{a-i}z^i\in J$ for any $a-b+2\leq i\leq a.$ Now if $a=b-1$ then $\gamma=c-1$ and by using $(y+z)^a\in J$, we get $y^a\in J$. If $a\geq b$, then as $y^b-(y+z)^{b-\gamma}z^\gamma\in J$, we get 
$$ (y^b-(y+z)^{b-\gamma}z^\gamma)y^{a-b-i}z^i\in J$$
for any $0\leq i\leq a-b$. It is similar to the proof of Subcase 2 of Case 1, by using the conditions 
$$(y+z)^a\in J\quad \text{and}\quad  (y^b-(y+z)^{b-\gamma}z^\gamma)y^{a-b-i}z^i\in J$$
for any $0\leq i\leq a-b$ and the Newton binomial expansion formula, we obtain a system of $(a-b+2)$ equations in $(a-b+2)$ indeterminates $y^{a}, y^{a-1}z,\ldots, y^{b-1}z^{a-b+1}$ as follows
\begin{align*}
M X=U,
\end{align*}
where $X^t =\bmt{y^{a} & y^{a-1}z&\cdots& y^{b-1}z^{a-b+1}},\; U^t=\bmt{u_1&u_2&\cdots &u_{a-b+2}}$ with  $u_i\in J$ and
the coefficient matrix $M\in  \MS_{a-b+2}$. It follows that $\det(M)>0$ by Lemma~\ref{Lemmadeterminant}. Thus $y^{a-i}z^i\in J$ for any $0\leq i\leq a-b+1$ by Cramer's rule.
\end{proof}

\begin{Rem}\label{Remark3.5}
[{\bf Matrix method}] In the proof of the above theorem, to prove that $y^{k+1-i}z^i\in J$ for all $0\leq i\leq k+1,$ we first look for a subset $\mathcal{I}\subset  \mathcal{J}:=\{0,1,\ldots,k+1\}$  such that $y^{k+1-i}z^i\in J$ for every $i\in \mathcal{I}$. We then use the method of taking the Newton binomial expansion formula of known elements of $J$ and represent their relations in a square system of equations in indeterminates $y^{k+1-i}z^i$ for every $i\in  \mathcal{J}\setminus  \mathcal{I}$ with the known elements in $J$ such that its coefficients matrix is a non-singular matrix. Thus, $y^{k+1-i}z^i\in J$ for any $ i\in  \mathcal{J}\setminus \mathcal{I}$ by Cramer's rule. For simplicity, we call this process \textbf{matrix method}.
\end{Rem}

We now prove the second main result in this section.
\begin{Theo} \label{Theorem3.8}  
Let  $I$ be as in Setting~\ref{IdealGorenstein}.  If $a\leq 2b-c$, then $R/I$ has the WLP for any $|a-b|+c-1\leq \beta \leq b-1$. 
\end{Theo}
 Notice that as $a\geq c$, the condition $a\leq 2b-c$ implies that $b\geq c.$
\begin{proof}
Set $k=\lfloor\frac{a+b+c-\beta-3}{2}\rfloor$ and  $L=x-y-z$.  By (ii) and (iii) of Remark~\ref{Remark3.3}, it suffices to show  that  if $|a-b|+c-1\leq \beta \leq b-1$, then 
$y^{k+1-i}z^i  \in J$  for any $0\leq i\leq c-1,$ 
where 
$J=((y+z)^a, y^b-(y+z)^{b-\gamma} z^{\gamma}, z^{c}, (y+z)^{a-b+\gamma}y^{b-\beta}, y^{b-\beta}z^{c-\gamma})$. 
To do it, again by  Remark~\ref{Remark3.3}(i), it is enough to prove that $y^{k+1-i}z^i  \in J$  for any $0\leq i\leq c-1$, provided   $|a-b|+c\leq \beta \leq b-1$ and $a+b+c-\beta$ is even; or $\beta=|a-b|+c-1\leq b-1$.

\noindent{\bf \underline{Case 1:  $|a-b|+c\leq \beta \leq b-1$ and $a+b+c-\beta$ is even.}} One has  $k=\frac{a+b+c-\beta-4}{2}.$ Now we go along the same lines as in the proof of Theorem~\ref{Theorem3.7}, we get $y^{k+1-i}z^i\in J$ for any $c-\gamma\leq i\leq c-1$. Thus, it remains to show that  $y^{k+1-i}z^i\in J$ for each $0\leq i\leq c-\gamma-1.$ 
This is proved by applying the matrix method in Remark~\ref{Remark3.5}  with the Newton binomial expansion formula of $(y+z)^{a-b+\gamma}y^{b-\beta}\in J$. Indeed, we have
\begin{align*}
\sum_{j=0}^{a-b+\gamma} \binom{a-b+\gamma}{j}y^{a-\beta+\gamma-j}z^j= (y+z)^{a-b+\gamma}y^{b-\beta} \in J.
\end{align*}
Notice that $k+1-(a-\beta+\gamma)-(c-\gamma-1) \geq 0.$ For each $0\leq i\leq c-\gamma-1$, 
\begin{align*}
\sum_{j=0}^{a-b+\gamma} \binom{a-b+\gamma}{j}y^{k+1-i-j}z^{i+j}= (y+z)^{a-b+\gamma}y^{b-\beta} y^{k+1-(a-\beta+\gamma)-i}z^i \in J.
\end{align*}
This gives a system of $(c-\gamma)$ equations with $(c-\gamma)$ indeterminates $y^{k+1}, y^{k}z,\ldots, y^{k+2-c+\gamma}z^{c-\gamma-1}$ as follows
 $$MX=U$$
 where $X^t=\bmt{y^{k+1}&\cdots& y^{k+2-c+\gamma}z^{c-\gamma-1}},\; U^t=\bmt{u_1&\cdots&u_{c-\gamma}}$ with $u_i\in J$ and the coefficient  matrix $M$ is an upper triangular matrix such that the entries on the main diagonal are all 1. Thus, by Cramer's rule,  $y^{k+1}, y^{k}z,\ldots, y^{k+2-c+\gamma}z^{c-\gamma-1}\in J.$

 \noindent{\bf \underline{Case 2:  $\beta=|a-b|+c-1\leq b-1$.}} In this case, one has $k+1=\min\{a,b\}\geq c$. 
 Now, for any $c-\gamma\leq i\leq c-1$, one has 
 \begin{align*}
 y^{k+1-i}z^i = y^{b-\beta}z^{c-\gamma}y^{k+1-b+\beta-i}z^{i-c+\gamma}\in J,
 \end{align*}
 as $k+1-b+\beta-i\geq \min\{a,b\}-b+|a-b|\geq 0.$    By applying the matrix method as in the proof of Case 1, we get $y^{k+1-i}z^i\in J$ for each $0\leq i\leq c-\gamma-1.$
\end{proof}

\begin{Rem} 
We now discuss cases that are not covered by Theorems~\ref{Theorem3.7} and~\ref{Theorem3.8}.
\begin{enumerate} 
\item[(a)] {\bf\underline{ Case 1:  $a\leq b$.}} Theorem~\ref{Theorem3.8} covers Theorem~\ref{Theorem3.7}; in other words, under this condition, we can omit the assumption $\gamma\geq\lfloor \frac{\beta-a+b+c-2}{2}\rfloor$ in Theorem~\ref{Theorem3.7}. Moreover, the cases $b\geq a\geq c$ and $1\leq \beta\leq b+c-a-2$ are not covered by Theorem~\ref{Theorem3.8} (and hence also by Theorem~\ref{Theorem3.7}). The first cases that are not covered by Theorem~\ref{Theorem3.8} correspond to $a=b=7, c=6, 1\leq \beta\leq 3$ and $1\leq \gamma\leq 5$. However, by using {\tt Macaulay2} \cite{Macaulay2}, we can check that the corresponding Gorenstein algebras have the WLP. 
\item[(b)]  {\bf\underline{ Case 2:  $a> b$.}} The cases $1\leq \beta\leq a-b+c-2$ and $1\leq \gamma\leq \lfloor \frac{\beta-a+b+c-4}{2}\rfloor$ are not covered by Theorems~\ref{Theorem3.7} and~\ref{Theorem3.8}. For example, the case $a=8, b=7, c=6, \beta=3$ and $\gamma=2$ is not covered Theorems~\ref{Theorem3.7} and~\ref{Theorem3.8}.  However, by using {\tt Macaulay2} \cite{Macaulay2}, we have checked that the corresponding Gorenstein algebra has the WLP.  
\end{enumerate}
\end{Rem}

As an application of  Theorems~\ref{Theorem3.7} and~\ref{Theorem3.8}, we can recover our result in \cite[Theorem 3.7]{MRT2019} in an easier way. 
\begin{Cor}\label{Corollary3.6}
Let  $I$ be as in Setting~\ref{IdealGorenstein}. If one of $a,b,c$ is equal to two, then $R/I$ has the WLP.
\end{Cor}

 The next result  gives a shorter and easier proof for our result in  \cite[Theorem ~3.15]{MRT2019}.
\begin{Cor}  \label{Corolarry3.9}
Let  $I$ be as in Setting~\ref{IdealGorenstein}. If one of $a,b,c$ is equal to three, then $R/I$ has the WLP.
\end{Cor}
\begin{proof}
It is assumed that $a\geq c$. It suffices to prove the theorem for the cases where $b=3$ or $c=3$.
	
\noindent{\bf \underline{Case 1: $b=3$.}}
If $a\geq c+1$, then  $R/I$ has the WLP by Theorem~\ref{Theorem3.7}. It remains to show that $R/I$ has the WLP whenever $a=c.$ Therefore, by symmetry of $x$ and $z$, we can assume $\gamma=2$. Since the socle degree of $R/I$ is $2a-\beta$, by Remark~\ref{Remark3.3}(i), it suffices to show that $R/I$ has the WLP for $\beta=1$. Set $k=a-1$ and $L=x-y-z$. We again apply the items (ii) and (iii) of Remark~\ref{Remark3.3}, it is enough to see that $y^{a-i}z^i\in J$ for all $0\leq i\leq a-1,$
where  $J=((y+z)^a, y^3-yz^2-z^3, z^a, (y+z)^{a-1}y^{2}, y^{2}z^{a-2}).$ 
Now applying the matrix method in Remark~\ref{Remark3.5} with elements $y^3-yz^2-z^3, y^2z^{a-2},(y+z)^a\in J$,  we get a system of $a$ equations with $a$ indeterminates $y^{a}, y^{a-1}z,\ldots, yz^{a-1}$ as follows
\begin{align*}
M X=U,
\end{align*}
where $X^t =\bmt{y^{a}&  \cdots &yz^{a-1}},\; U^t=\bmt{u_1&\cdots &u_{a}}$ with  $u_i\in J$ and 
$M\in \MS_{a}$. Hence, by Lemma~\ref{Lemmadeterminant}, $\det(M)>0$ and we get $y^{a-i}z^i\in J$ for any $0\leq i\leq a-1$.

 \noindent{\bf \underline{Case 2: $c=3$.}}
By Theorems~\ref{Theorem3.7} and~\ref{Theorem3.8},  $R/I$ has the WLP, whenever  $a\geq b+1$. Therefore we  consider the case where $a=b$ or $a=b-1$.  Set $k=\lfloor\frac{a+b-\beta}{2}\rfloor$ and $L=x-y-z$. By (ii) and (iii) of Remark~\ref{Remark3.3}, it is enough to see that $y^{k+1-i}z^i\in J$ for all $0\leq i\leq 2.$ However, this is simply proved by using the matrix method in Remark~\ref{Remark3.5}.
\end{proof}

We close this section by giving the following corollary.
\begin{Cor}\label{Corollary3.13}
Let  $I$ be as in Setting~\ref{IdealGorenstein}. Then $R/I$ has the WLP if one of the following conditions satisfies:
\begin{enumerate}
\item [\rm (a)] $a\geq 2b+c-6$;
\item [\rm (b)] $a\geq b+c-2$ and $1\leq \beta\leq a-b-c+5.$
\end{enumerate}
\end{Cor}
\begin{proof}
If $b\geq 4$ and $a\geq 2b+c-6$, then $b-a+c-1\leq 1$ and $\lfloor\frac{\beta-a+b+c-2}{2}\rfloor\leq 1$. It follows from Theorem~\ref{Theorem3.7} that $R/I$ has the WLP. Furthermore, if $b=2$ or $b=3$, then $R/I$ has the WLP by Corollaries~\ref{Corollary3.6} and~\ref{Corolarry3.9}, respectively. Thus, we completely prove  (a).

For the item (b), if $1\leq \beta\leq a-b-c+5$, then $\lfloor\frac{\beta-a+b+c-2}{2}\rfloor \leq 1.$ Hence $R/I$ has the WLP for any $\gamma\geq 1$ by Theorem~\ref{Theorem3.7}.
\end{proof}

\section{One of $a,b,c$ is equal to $4$} \label{Section4}
It is assumed that $a\geq c$. Therefore, we first consider the case where $c=4.$ We have the following.
\begin{pro} \label{Proposition4.1}
Assume that  $ a,b\geq 4, a\geq b-2$  and $\max\{1,b-a+1\}\leq \gamma\leq 3$. For any  $1\leq \beta\leq  b-1$, the ideal
$$I=(x^a, y^b-x^{b-\gamma} z^{\gamma}, z^4, x^{a-b+\gamma}y^{b-\beta}, y^{b-\beta}z^{4-\gamma})$$
has the WLP.
\end{pro}
\begin{proof}
By Corollary~\ref{Corollary3.13}(a), $R/I$ has the WLP for any $a\geq 2b-2.$ Thus, we only consider the following two cases.
	
\noindent{\bf \underline{Case 1: $b+2\leq a\leq 2b-3$.}} It follows from Corollary~\ref{Corollary3.13}(b) that $R/I$ has the WLP for any $1\leq \beta\leq a-b+1.$ It remains to show that $R/I$ has the WLP for any $a-b+2\leq \beta\leq b-1.$ We first observe that if $a-b+3\leq \beta\leq b-1$, then  $b+2\leq a\leq 2b-4$. Therefore,  $R/I$ has the WLP for any $a-b+3\leq \beta\leq b-1$ by Theorem~\ref{Theorem3.8}. Thus, it suffices to show that $R/I$ has the WLP for 
$\beta=a-b+2.$ By Theorem~\ref{Theorem3.7}, $R/I$ has the WLP for any $\gamma\geq \lfloor\frac{\beta-a+b+2}{2}\rfloor =2.$ Then, it is enough to consider the case where $\gamma=1.$  In this case, we set $k=b-1$ and $L=x-y-z$. By Remark~\ref{Remark3.3}(iii), 
it remains to show that 
$y^{b-i}z^i\in J$ for any $0\leq i\leq 3,$
where $J=((y+z)^a, y^b-(y+z)^{b-1} z, z^4, (y+z)^{a-b+1}y^{b-\beta}, y^{b-\beta}z^{3}).$ This is immediately proved by using the matrix method.


\noindent{\bf \underline{Case 2: $b-2\leq a\leq b+1$.}} We first observe that,  by Theorem~\ref{Theorem3.8},  $R/I$ has the WLP for any $|a-b|+3\leq \beta\leq b-1$. By Remark~\ref{Remark3.3}(i), it suffices to show that $R/I$ has the WLP, whenever $1\leq \beta\leq |a-b|+2$ and $a+b-\beta$ is even. In this case, we set $k=\frac{a+b-\beta}{2}$ and $L=x-y-z$. Again by (ii) and (iii) of Remark~\ref{Remark3.3}, it remains to show that 
$$y^{k+1}, y^kz, y^{k-1}z^2, y^{k-2}z^3\in J,$$
where  
$J=((y+z)^a, y^b-(y+z)^{b-\gamma} z^{\gamma}, z^4, (y+z)^{a-b+\gamma}y^{b-\beta}, y^{b-\beta}z^{4-\gamma}).$

It is proved by considering the four cases where $a=b+i,\ i\in \{-2,\ldots,1\}$ and applying the matrix method in Remark~\ref{Remark3.5}. 
\end{proof}

Now we consider the case where $b=4$. We have the following result.
\begin{pro}\label{Proposition4.2}
Let  $I$ be as in Setting~\ref{IdealGorenstein}. Assume  $ 1\leq \beta, \gamma\leq 3$  and $a\geq c\geq 4$. Then the ideal
\begin{align*}
I=(x^a, y^4 - x^{4-\gamma} z^\gamma, z^c, x^{a-4+\gamma}y^{4-\beta}, y^{4-\beta}z^{c-\gamma})
\end{align*}
has the WLP.
\end{pro}
\begin{proof}
If $a\geq c+2$, then $R/I$ has the WLP by Corollary~\ref{Corollary3.13}(a). Thus, we only consider the following two cases.

\noindent{\bf \underline{Case 1: $a=c+1$.}} More precisely, we consider the ideal
\begin{align*}
I=(x^{c+1}, y^4-x^{4-\gamma}z^{\gamma}, z^c, x^{c-3+\gamma}y^{4-\beta}, y^{4-\beta}z^{c-\gamma})
\end{align*}
where  $1\leq \beta,\gamma \leq 3$ and $c\geq 4.$  Notice that the socle degree of $R/I$ is $D=2c+2-\beta$. By Remark~\ref{Remark3.3}(i), it is enough to prove that $R/I$ has the WLP whenever $\beta=1$ or $\beta=3$. We will consider the following two subcases. 

\noindent{\bf \underline{Subcase 1: $\beta=1$.}} In this case, we set $k=c$ and $L=x-y-z$. By (ii) and (iii) of Remark~\ref{Remark3.3}, it is enough to show that $y^{c+1-i}z^i\in J$ for any $0\leq i\leq c-1$,
where $$J=((y+z)^{c+1}, y^4-(y+z)^{4-\gamma} z^\gamma, z^c, (y+z)^{c-3+\gamma}y^{3}, y^{3}z^{c-\gamma}).$$
Firstly, we consider the case where $2\leq \gamma\leq 3$. Then
\begin{align*}
y^3z^{c-2}=   y^3z^{c-\gamma} z^{\gamma-2}\in J.
\end{align*}
Now since $y^4-(y+z)^{4-\gamma} z^\gamma,y^3z^{c-2}, (y+z)^{c+1}\in J$, we get a system of $c$ equations in $c$ indeterminates $y^{c+1}, y^{c}z,\ldots, y^2z^{c-1}$ as follows
\begin{align*}
M X=U,
\end{align*}
where $X^t =\bmt{y^{c+1}& y^{c}z& \cdots &y^2z^{c-1}},\; U^t=\bmt{u_1&u_2&\cdots &u_c}$ with  $u_i\in J$ and 
$M\in \MS_{c}$. By Lemma~\ref{Lemmadeterminant}, $\det(M)> 0$. This completes the proof of the case where $2\leq \gamma\leq 3$.

Now we consider the case where $\gamma=1.$ Let $L$ be a general linear form. It is enough to show that the map 
$$\times L: [R/I]_{c} \longrightarrow [R/I]_{c+1}$$
is an isomorphism. Set $\point=(x^{c+1}, y^4-x^{3}z, z^c, x^{c-2}y^{3})$ and $ \aa =(x^{c+1}, y^4-x^{3}z, z^c)$. Clearly, $I_c=\point_c,\ I_{c+1}=\point_{c+1}$. Set  $\mathfrak{b}=\aa\colon_R \point$. By \cite[Proposition~5.2]{BE1977},  $R/\bb$ is a Gorenstein algebra  and  we have an exact sequence
\begin{align*}
\xymatrix{ 0\ar[r]& (R/\bb)(-c-1)\ar[rr]^{\qquad \times x^{c-2}y^{3}}&&R/\aa \ar[r] & R/\point \ar[r]&0 }
\end{align*}
which gives the following commutative diagram
\begin{align*}
\xymatrix{ 0\ar[r]& 0\ar[rr]\ar[d]&&[R/\aa]_c  \ar[r]\ar[d]&[R/I]_c  \ar[r]\ar[d]^{\times L} &0  \\ 
	 0\ar[r]& [R/\bb]_0=K\ar[rr]^{\times x^{c-2}y^3}\ar[d]&&[R/\aa]_{c+1}  \ar[r]\ar[d] &[R/I]_{c+1}  \ar[r]\ar[d] &0 \\  
	 &  K\ar[rr] &&[R/(\aa,L)]_{c+1}  \ar[r] &[R/(I,L)]_{c+1}   &   }.
\end{align*}
Since the socle degree of $R/\aa$ is $2c+2$ and $R/\aa$ has the WLP, the map $\times L: [R/\aa]_c \longrightarrow [R/\aa]_{c+1}$ is injective. By the Snake lemma, we get an exact sequence
\begin{align*}
0\longrightarrow \Ker(\times L) \longrightarrow K \longrightarrow [R/(\aa,L)]_{c+1} \longrightarrow  [R/(I,L)]_{c+1} \longrightarrow 0.
\end{align*}
Further, $R/\aa$ is resolved by the Koszul complex, hence
\begin{align*}
\dim_K[R/(\aa,L)]_{c+1}&=H_{R/\aa}(c+1)-H_{R/\aa}(c)=1.
\end{align*}
Thus the map  $K \longrightarrow [R/(\aa,L)]_{c+1}$, $1\longmapsto x^{c-2}y^{3}$ is an isomorphism. It follows that 
the map $\times L: [R/I]_{c} \longrightarrow [R/I]_{c+1}$
is an isomorphism.

\noindent{\bf \underline{Subcase 2: $ \beta=3$.}} By Theorem~\ref{Theorem3.7}, $R/I$ has the WLP for any $\gamma\geq 2$. It suffices to show that $R/I$ has the WLP for $\gamma=1.$  Let $L$ be a general linear form. We have to show that  
$$\times L: [R/I]_{c-1} \longrightarrow [R/I]_{c}$$
is an isomorphism. 
Set $\point=(x^{c+1}, y^4-x^{3}z, z^c, x^{c-2}y), \aa =(x^{c+1}, y^4-x^{3}z, z^c)$ and $\mathfrak{b}=\aa\colon_R \point$. Then $R/\bb$ is a Gorenstein algebra and we have an exact sequence
\begin{align*}
\xymatrix{ 0\ar[r]& (R/\bb)(-c+1)\ar[rr]^{\qquad \times x^{c-2}y}&&R/\aa \ar[r] & R/\point \ar[r]&0 }
\end{align*}
which gives the following commutative diagram
\begin{align*}
\xymatrix{ 0\ar[r]&  [R/\bb]_0=K \ar[rr]^{\times x^{c-2}y}\ar[d]_{\times L}&&[R/\aa]_{c-1}  \ar[r]\ar[d]^{\times L} &[R/\point]_{c-1}  \ar[r]\ar[d]^{\times L} &0  \\   
	0\ar[r]& [R/\bb]_1=[R]_1\ar[rr]^{\times x^{c-2}y} \ar[d]&&[R/\aa]_{c}  \ar[r]\ar[d] &[R/\point]_{c}  \ar[r]\ar[d] &0\\  &  [R/(\bb,L)]_1\ar[rr] &&[R/(\aa,L)]_{c}  \ar[r] &[R/(\point,L)]_{c}   &   }.
\end{align*}
Again the map  $\times L: [R/\aa]_{c-1} \longrightarrow [R/\aa]_{c}$ is injective and
\begin{align*}
\dim_K[R/(\aa,L)]_{c}&=H_{R/\aa}(c)-H_{R/\aa}(c-1)=3.
\end{align*}
Denote by $Q$ the kernel of  $\times L:[R/\point]_{c-1}\longrightarrow [R/\point]_{c}$. By the Snake lemma, we get an exact sequence
\begin{align*}
\xymatrix{0\ar[r]&  Q\ar[r] & [R/(\bb,L)]_1 \ar[rr]^{\times x^{c-2}y} &&  [R/(\aa,L)]_{c}\ar[r] &[R/(\point,L)]_{c}\ar[r]&0  }.
\end{align*}
Notice that $\dim_K[R/(\bb,L)]_1=2$. It follows that the map  
$$\times x^{c-2}y: [R/(\bb,L)]_1 \longrightarrow [R/(\aa,L)]_{c}$$  is injective, so $Q=0$. In other words, the map 
$\times L: [R/\point]_{c-1} \longrightarrow [R/\point]_{c}$
is injective. Now we consider the following  commutative diagram, where the first two rows are exact
\begin{align*}
\xymatrix{ 0\ar[r]& 0\ar[rr]\ar[d]&&[R/\point]_{c-1}  \ar[r]\ar[d]^{\times L} &[R/I]_{c-1}  \ar[r]\ar[d]^{\times L} &0  \\  
	 0\ar[r]& [R/\point]_0=K\ar[rr]^{\times y z^{c-1}} \ar[d]&&[R/\point]_{c}  \ar[r]\ar[d] &[R/I]_{c}  \ar[r]\ar[d] &0\\  &  K\ar[rr] &&[R/(\point,L)]_{c}  \ar[r] &[R/(I,L)]_{c}   &   }.
\end{align*}
Denote by $\Omega$ the kernel of  $\times L:[R/I]_{c-1}\longrightarrow [R/I]_{c}$. As the map 
$$\times L: [R/\point]_{c-1}\longrightarrow [R/\point]_c$$ is injective, we get an exact sequence
\begin{align*}
\xymatrix{0\ar[r]&  \Omega\ar[r] & K \ar[rr]^{\times yz^{c-1}\qquad} &&  [R/(\point,L)]_{c}\ar[r] &[R/(I,L)]_{c}\ar[r]&0  }.
\end{align*}
Notice that the map $\times yz^{c-1}:K \longrightarrow [R/(\point,L)]_{c}$ is injective, hence $\Omega=0$. It completes the proof in this subcase.

\noindent{\bf \underline{Case 2: $a= c$.}} In this case, the socle degree of $R/I$ is $2c+1-\beta$. By Remark~\ref{Remark3.3}(i), it suffices to show that $R/I$ has the WLP for $\beta=2.$ By the symmetry of $x$ and $z$, we can assume $2\leq \gamma\leq 3.$ We consider the two subcases. 

\noindent{\bf \underline{Subcase 1: $\gamma=3$.}} Set $k=c-1$ and $L=x-y-z$. By (ii) and (iii) of Remark~\ref{Remark3.3}, it remains to prove that $y^{c-i}z^i\in J$ for all $0\leq i\leq c-1$, 
where $$J=((y+z)^{c}, y^4-(y+z) z^3, z^c, (y+z)^{c-1}y^{2}, y^{2}z^{c-3}).$$
Now applying the matrix method in Remark~\ref{Remark3.5}, with elements $y^4-(y+z) z^3, y^2z^{c-3}$ and $(y+z)^{c}$ belong to $J$, we get a system of $c$ equations in $c$ indeterminates $y^{c}, y^{c-1}z,\ldots, yz^{c-1}$ with its coefficient matrix $M\in \MS_c$.   By Lemma~\ref{Lemmadeterminant}, $\det(M)>0$. By Cramer's rule, $y^{c-i}z^i\in J$ for all $0\leq i\leq c-1$.

\noindent{\bf \underline{Subcase 2: $\gamma=2$.}} Let $L$ be a general linear form. It is enough to show that 
$$\times L: [R/I]_{c-1} \longrightarrow [R/I]_{c}$$
is an isomorphism. 
Set $\point=(x^{c},y^4-x^{2}z^2, z^c, x^{c-2}y^2), \aa =(x^{c}, y^4-x^{2}z^2, z^c)$. We go along the same lines as in the proof of Subcase 2 of Case 1, we get $\times L:[R/I]_{c-1}\longrightarrow [R/I]_{c}$ is injective. So, we have completed the proof.
\end{proof}


We close this section by stating our main result as follows.
\begin{Theo}  \label{Theorem4.3}
Let  $I$ be as in Setting~\ref{IdealGorenstein}. If one of $a,b,c$ is equal to four, then $R/I$ has the WLP.
\end{Theo}
\begin{proof}
As $a\geq c$, we only consider the cases where $b=4$ or $c=4$. The theorem implies from Propositions~\ref{Proposition4.1} and~\ref{Proposition4.2}.
\end{proof}

\section{One of $a,b,c$ is equal to $5$} \label{Section5}
Recall that we always assume that $a\geq c$. Firstly,  we consider the case  where $c=5.$ We have the following result.
\begin{pro} \label{Proposition5.1}
Assume that  $ a,b\geq 5, a\geq b-3$  and $\max\{1,b-a+1\}\leq \gamma\leq 4$. For any  $1\leq \beta\leq  b-1$, the ideal
$$I=(x^a, y^b-x^{b-\gamma} z^{\gamma}, z^5, x^{a-b+\gamma}y^{b-\beta}, y^{b-\beta}z^{5-\gamma})$$
has the WLP.
\end{pro}
\begin{proof}
 Set $k=\lfloor\frac{a+b-\beta+2}{2}\rfloor$ and $L=x-y-z$. By (ii) and (iii) of Remark~\ref{Remark3.3},  it is enough to prove the following assertion:
 
 \noindent \textsc{ Claim:} $y^{k+1-i}z^i\in J$ for any $0\leq i\leq 4,$ where 
 $$J=((y+z)^a, y^b-(y+z)^{b-\gamma} z^\gamma, z^5, (y+z)^{a-b+\gamma}y^{b-\beta}, y^{b-\beta}z^{5-\gamma}).$$
By Corollary~\ref{Corollary3.13}(a), $R/I$ has the WLP for any $a\geq 2b-1.$ Thus, we only consider the following three cases.
	
\noindent{\bf \underline{Case 1: $b+3\leq a\leq 2b-2$.}} It follows from Corollary~\ref{Corollary3.13}(b) that $R/I$ has the WLP for any $1\leq \beta\leq a-b.$ It remains to show that $R/I$ has the WLP for any $a-b+1\leq \beta\leq b-1.$ Firstly, if $a-b+4\leq \beta\leq b-1$, then  $b+3\leq a\leq 2b-5$ and therefore $R/I$ has the WLP by Theorem~\ref{Theorem3.8}. By Remark~\ref{Remark3.3}(i), it suffices to prove the above claim, whenever $\beta=a-b+1$ or $\beta=a-b+3$. 
To do it, we will consider the following subcases.

\noindent{\bf \underline{Subcase 1: $\beta=a-b+1$.}}
Firstly,  $R/I$ has the WLP for any $\gamma\geq \lfloor\frac{\beta-a+b+3}{2}\rfloor =2$ by Theorem~\ref{Theorem3.7}. Thus, we only consider the case where $\gamma=1.$  In this case, one has $k=b$. As $y^{b-\beta}z^{4}\in J$ and $\beta\geq 4$, we get $y^{b-3}z^4 \in J.$
By applying the matrix method in Remark~\ref{Remark3.5}, with elements $y^b-(y+z)^{b-1} z$ and $(y+z)^{\beta}y^{b-\beta}$ belong to $J$, it is easy  to see that 
$$y^{b+1-i}z^i\in J\quad  \text{for any} \; 0\leq i\leq 3.$$

\noindent{\bf \underline{Subcase 2: $\beta=a-b+3$.}}
Then $R/I$ has the WLP for any $\gamma\geq \lfloor\frac{\beta-a+b+3}{2}\rfloor =3$ by Theorem~\ref{Theorem3.7}. Thus, we only consider the cases where $\gamma=1$ or $\gamma=2.$  In this case, one has $k=b-1$. Analogously, by using the matrix method in Remark~\ref{Remark3.5}, we get $y^{b-i}z^i\in J$, for all $0\leq i\leq 4.$
	
\noindent{\bf \underline{Case 2: $b\leq a\leq b+2$.}}
By Theorem~\ref{Theorem3.8}, $R/I$ has the WLP for $a-b+4\leq \beta\leq b-1$. Therefore, by Remark~\ref{Remark3.3}(i), we will prove the claim for the cases where $1\leq \beta\leq a-b+3$ and $a+b-\beta$ is odd. However, this can be solved by applying the matrix method in Remark~\ref{Remark3.5} for  which choices of parameters $a,b,\beta$ and $\gamma.$

\noindent{\bf \underline{Case 3: $b-3\leq a\leq b-1$.}}	By Theorem~\ref{Theorem3.8}, $R/I$ has the WLP for $b-a+4\leq \beta\leq b-1$. By Remark~\ref{Remark3.3}(i), we will prove the above claim under the conditions $1\leq \beta\leq b-a+3$ and $a+b-\beta$ is odd. Analogously, the above claim is completely shown by applying the matrix method in Remark~\ref{Remark3.5}.
\end{proof}


Now we consider the case where $b=5$. We have the following proposition.
\begin{pro}\label{Proposition5.2}
Let  $I$ be as in Setting~\ref{IdealGorenstein}. Assume  $ 1\leq \beta, \gamma\leq 4$  and $a\geq c\geq 5$. Then the ideal
\begin{align*}
I=(x^a, y^5 - x^{5-\gamma} z^\gamma, z^c, x^{a-5+\gamma}y^{5-\beta}, y^{5-\beta}z^{c-\gamma})
\end{align*}
has the WLP.
\end{pro}
\begin{proof}
If $a\geq c+4$, then $R/I$ has the WLP by Corollary~\ref{Corollary3.13}(a). By Theorem~\ref{Theorem3.7}, $R/I$ has the WLP, whenever $\max\{1,c-a+4\}\leq \beta\leq 4$ and $\gamma\geq\lfloor\frac{\beta+c-a+3}{2}\rfloor.$ Thus, we only consider the following four cases.

\noindent{\bf \underline{Case 1: $a=c$.}}  Notice that the socle degree of $R/I$ is $D=2a+2-\beta$. By Remark~\ref{Remark3.3}(i), it suffices to show that $R/I$ has the WLP, provided $\beta=1$  or $\beta=3$. By symmetry of $x$ and $z$, we can assume that $3\leq \gamma\leq 4$. Thus we consider the following  subcases.

\noindent{\bf \underline{Subcase 1: $\beta=1$ and $3\leq \gamma \leq 4$.}} 
More precisely, we consider the ideal
\begin{align*}
I=(x^{a}, y^5 - x^{5-\gamma} z^\gamma, z^{a}, x^{a-5+\gamma}y^4, y^4z^{a-\gamma}).
\end{align*}
Let $L$ be a general linear form. It suffices to show that the map
$$\times L: [R/I]_{a}\longrightarrow [R/I]_{a+1}$$
is an isomorphism. Set $\point=(x^{a}, y^5-x^{5-\gamma}z^\gamma, z^{a}, y^4z^{a-\gamma}), \aa =(x^{a}, y^5-x^{5-\gamma}z^\gamma, z^{a})$ and $\mathfrak{b}=\aa\colon_R \point$. Then $I_{a+1}=\point_{a+1},\ I_{a}=\point_{a}$ and $R/\bb$ is Gorenstein. Moreover, we have an exact sequence
\begin{align*}
\xymatrix{ 0\ar[r]& (R/\bb)(-a-4+\gamma)\ar[rrr]^{ \qquad\times y^4z^{a-\gamma}}&&&R/\aa \ar[r] & R/\point \ar[r]&0 }
\end{align*}
which gives the following commutative diagram
\begin{align*}
\xymatrix{ 0\ar[r]& [R/\bb]_{\gamma-4}\ar[rr]\ar[d]&&[R/\aa]_{a}  \ar[r]\ar[d]&[R/I]_{a}  \ar[r]\ar[d]^{\times L} &0  \\ 
	0\ar[r]& [R/\bb]_{\gamma-3}\ar[rr]^{\times y^4z^{a-\gamma}}\ar[d]&&[R/\aa]_{a+1}  \ar[r]\ar[d] &[R/I]_{a+1}  \ar[r]\ar[d] &0 \\  
	&  [R/(\bb,L)]_{\gamma-3}\ar[rr] &&[R/(\aa,L)]_{a+1}  \ar[r] &[R/(I,L)]_{a+1}   &   }.
\end{align*}
As $R/\aa$ has the WLP, hence $\times L: [R/\aa]_{a-1} \longrightarrow [R/\aa]_{a}$ is injective. By the Snake lemma, we get an exact sequence
\begin{align}\label{Sequence5.1}
0\longrightarrow \Ker(\times L) \longrightarrow [R/(\bb,L)]_{\gamma-3} \longrightarrow [R/(\aa,L)]_{a+1} \longrightarrow  [R/(I,L)]_{a+1} \longrightarrow 0.
\end{align}
Further, since $R/\aa$ is resolved by the Koszul complex and it  has the WLP, we get
\begin{align*}
\dim_K[R/(\aa,L)]_{a+1}&=H_{R/\aa}(a+1)-H_{R/\aa}(a)=1.
\end{align*}
If $\gamma=3$, then the map  $\times y^4z^{a-3}: [R/(\bb,L)]_{0}=K \longrightarrow [R/(\aa,L)]_{a+1}\cong K$  is an isomorphism. If $\gamma=4$, then as $y^5z^{a-4}=xz^{a}\in \aa$, we get $y\in \bb$. It follows that $\dim_K[R/\bb]_{1}=2$, hence $\dim_K[R/(\bb,L)]_{1}=1.$ Thus the map  
$$\times y^4z^{a-4}: [R/(\bb,L)]_{1}=K \longrightarrow [R/(\aa,L)]_{a+1}$$
is also an isomorphism. In summary, by \eqref{Sequence5.1} we get that the map
$$\times L: [R/I]_{a} \longrightarrow [R/I]_{a+1}$$
is always an isomorphism.

\noindent{\bf \underline{Subcase 2: $\beta=3$ and $\gamma = 3$.}} \label{Subcase2ofCase1}
More precisely, we consider the ideal
\begin{align*}
I=(x^{a}, y^5 - x^2z^3, z^{a}, x^{a-2}y^2, y^2z^{a-3}).
\end{align*}
Let $L$ be a general linear form. We will show that the map
$$\times L: [R/I]_{a-1}\longrightarrow [R/I]_{a}$$
is an isomorphism. Set $\point=(x^{a}, y^5 - x^2z^3, z^{a},  y^2z^{a-3}), \aa =(x^{a}, y^5 - x^2z^3, z^{a})$ and $\mathfrak{b}=\aa\colon_R \point$. Then $R/\bb$ is Gorenstein. Moreover, we have an exact sequence
\begin{align*}
\xymatrix{ 0\ar[r]& (R/\bb)(-a+1)\ar[rrr]^{ \qquad\times y^2z^{a-3}}&&&R/\aa \ar[r] & R/\point \ar[r]&0 }
\end{align*}
which gives the following commutative diagram
\begin{align*}
\xymatrix{ 0\ar[r]& [R/\bb]_{0}=K\ar[rr]\ar[d]&&[R/\aa]_{a-1}  \ar[r]\ar[d]&[R/\point]_{a-1}  \ar[r]\ar[d]^{\times L} &0  \\ 
	0\ar[r]& [R/\bb]_{1}=[R]_1\ar[rr]^{\times y^2z^{a-3}}\ar[d]&&[R/\aa]_{a}  \ar[r]\ar[d] &[R/\point]_{a}  \ar[r]\ar[d] &0 \\  
	&  [R/(\bb,L)]_{1}\ar[rr] &&[R/(\aa,L)]_{a}  \ar[r] &[R/(\point,L)]_{a}   &   }.
\end{align*}
As in the above proof, we have that $\times L: [R/\aa]_{a-1} \longrightarrow [R/\aa]_{a}$ is injective. By the Snake lemma, we get an exact sequence
\begin{align}\label{Sequence5.2}
0\longrightarrow \Ker(\times L) \longrightarrow [R/(\bb,L)]_{1} \longrightarrow [R/(\aa,L)]_{a} \longrightarrow  [R/(I,L)]_{a} \longrightarrow 0.
\end{align}
Moreover, we have $\dim_K[R/(\bb,L)]_{1}=2$ and
\begin{align*}
\dim_K[R/(\aa,L)]_{a}&=H_{R/\aa}(a)-H_{R/\aa}(a-1)=3.
\end{align*}
It follows that  $\times y^2z^{a-3}: [R/(\bb,L)]_{1}\longrightarrow [R/(\aa,L)]_{a}$
is injective. By \eqref{Sequence5.2} we get that $\times L: [R/\point]_{a-1}\longrightarrow [R/\point]_a$ is injective. Now we consider the following commutative diagram where the first two rows are exact
\begin{align*}
\xymatrix{ 0\ar[r]& 0\ar[rr]\ar[d]&&[R/\point]_{a-1}  \ar[r]\ar[d]&[R/I]_{a-1}  \ar[r]\ar[d]^{\times L} &0  \\ 
	0\ar[r]&[R/\point]_{0}=K\ar[rr]^{\times x^{a-2}y^2}\ar[d]&&[R/\point]_{a}  \ar[r]\ar[d] &[R/I]_{a}  \ar[r]\ar[d] &0 \\  
	&  K\ar[rr] &&[R/(\point,L)]_{a}  \ar[r] &[R/(I,L)]_{a}   &   }.
\end{align*}
It is easy to see that the map $\times  x^{a-2}y^2: K\longrightarrow [R/(\point,L)]_{a}$ is injective. By the Snake lemma, we get that $$\times L: [R/I]_{a-1}\longrightarrow [R/I]_{a}$$
is injective.

\noindent{\bf \underline{Subcase 3: $\beta=3$ and $\gamma = 4$.}} 
More precisely, we consider the ideal
\begin{align*}
I=(x^{a}, y^5 - x z^4, z^{a}, x^{a-1}y^2, y^2z^{a-4}).
\end{align*}
Set $L=x-y-z$. By (ii) and (iii) of Remark~\ref{Remark3.3}, it suffices to show that $y^{a-i}z^i\in J$ for all $0\leq i\leq a-1$,
where 
$$J=((y+z)^{a}, y^5-(y+z) z^4, z^{a}, (y+z)^{a-1}y^{2}, y^{2}z^{a-4}).$$
By using $y^5-yz^4 -z^5,y^{2}z^{a-4},(y+z)^{a}\in J,$ the matrix method in Remark~\ref{Remark3.5} gets a system of $a$ equations in $a$ indeterminates $y^{a}, y^{a-1}z,\ldots, yz^{a-1}$ as follows
$$MX=U$$
where $X^t=\bmt{y^{a} & y^{a-1}z &\cdots & yz^{a-1}}$ and $U^t=\bmt{u_1&u_2&\cdots&u_{a}}$ with $u_i\in J$ and $M\in \MS_{a}$. Since $\det(M)>0$ by Lemma~\ref{Lemmadeterminant}, $y^{a-i}z^i\in J$ for all $0\leq i\leq a-1$.


\noindent{\bf \underline{Case 2: $a=c+1$.}} It follows that $R/I$ has the WLP for any $3\leq \beta\leq 4$ and $\gamma\geq \lfloor\frac{\beta+2}{2}\rfloor$ by Theorem~\ref{Theorem3.7}. Notice that the socle degree of $R/I$ is $D=2a+1-\beta$. By Remark~\ref{Remark3.3}, we show that $R/I$ has the WLP, provided $\beta=2$ and $1\leq \gamma\leq 4$; or $\beta=4$ and $1\leq \gamma \leq 2$. Thus we consider the following five subcases.

\noindent{\bf \underline{Subcase 1: $\beta=2$ and $3\leq \gamma \leq 4$.}} 
More precisely, we consider the ideal
\begin{align*}
I=(x^{a}, y^5 - x^{5-\gamma} z^\gamma, z^{a-1}, x^{a-5+\gamma}y^3, y^3z^{a-1-\gamma}).
\end{align*}
Set $L=x-y-z$. Again by (ii) and (iii) of Remark~\ref{Remark3.3}, it remains to show that  $y^{a-i}z^i\in J$ for all $0\leq i\leq a-2$,
where $$J=((y+z)^{a}, y^5-(y+z)^{5-\gamma} z^\gamma, z^{a-1}, (y+z)^{a-5+\gamma}y^{3}, y^{3}z^{a-1-\gamma}).$$
As $y^{3}z^{a-3}, y^4z^{a-4}\in J$ since $y^{3}z^{a-1-\gamma}\in J$. Finally, by using $$y^5-(y+z)^{5-\gamma} z^\gamma,y^{4}z^{a-4}, y^{3}z^{a-3},(y+z)^{a}\in J,$$
the matrix method in Remark~\ref{Remark3.5} gets a system of $(a-1)$ equations in $(a-1)$ indeterminates $y^{a}, \ldots, y^{2}z^{a-2}$ as follows
$$MX=U$$
where $X^t=\bmt{y^{a} & \cdots & y^{2}z^{a-2}}$ and $U^t=\bmt{u_1&\cdots&u_{a-1}}$ with $u_i\in J$ and $M\in \MS_{a-1}$. It follows that $\det(M)>0$ by Lemma~\ref{Lemmadeterminant}. Thus, $y^{a-i}z^i\in J$ for all $0\leq i\leq a-2$.

\noindent{\bf \underline{Subcase 2: $\beta=2$ and $\gamma =1$.}} 
Let $L$ be a general linear form. Let us check that
$$\times L: [R/I]_{a-1}\longrightarrow [R/I]_{a}$$
is an isomorphism. Set $\point=(x^{a}, y^5-x^{4}z, z^{a-1}, x^{a-4}y^{3}), \aa =(x^{a}, y^5-x^{4}z, z^{a-1})$ and $\mathfrak{b}=\aa\colon_R \point$. Then $I_{a-1}=\point_{a-1},\ I_{a}=\point_{a}$. We now go along the same lines as in the proof of Subcase 1 of Case 1, we get that  $\times x^{a-4}y^3: [R/(\bb,L)]_1 \longrightarrow [R/(\aa,L)]_{a}$  is an isomorphism. Thus 
$$\times L: [R/I]_{a-1} \longrightarrow [R/I]_{a}$$
is an isomorphism.

\noindent{\bf \underline{Subcase 3: $\beta=2$ and $\gamma =2$.}} 
Let $L$ be a general linear form.  We use  the same method as in the proof of Subcase 2 of Case 1 to show that
$$\times L: [R/I]_{a-1}\longrightarrow [R/I]_{a}$$
is injective, by setting $\point=(x^{a}, y^5-x^{3}z^2, z^{a-1}, x^{a-3}y^{3})$ and $\aa =(x^{a}, y^5-x^{3}z^2, z^{a-1})$.

\noindent{\bf \underline{Subcase 4: $\beta=4$ and $\gamma =1$.}} 
Set $\point=(x^{a}, y^5 - x^{4} z, z^{a-1},  yz^{a-2})$ and $ \aa=(x^{a}, y^5 - x^{4} z, z^{a-1})$. Again we use  the same method as in the proof of Subcase 2 of Case 1 to show that
$$\times L: [R/I]_{a-2}\longrightarrow [R/I]_{a-1}$$
is injective, where $L$ is a general linear form.

\noindent{\bf \underline{Subcase 5: $\beta=4$ and $\gamma =2$.}} 
Set $\point=(x^{a}, y^5 - x^{3} z^2, z^{a-1}, x^{a-3}y)$ and $\aa=(x^{a}, y^5 - x^{3} z^2, z^{a-1}).$ Let $L$ be a general linear form.  We go along the same lines as in the proof of Subcase 2 of Case 1 to show that $\times L: [R/I]_{a-2} \longrightarrow [R/I]_{a-1}$ is injective. 

\noindent{\bf \underline{Case 3: $a=c+2$.}} It follows that $R/I$ has the WLP for any $2\leq \beta\leq 4$ and $\gamma\geq \lfloor\frac{\beta+1}{2}\rfloor$ by Theorem~\ref{Theorem3.7}. Notice that the socle degree of $R/I$ is $D=2a-\beta$. By Remark~\ref{Remark3.3}(i), it suffices to show that $R/I$ has the WLP, provided $\beta=1$ and $1\leq \gamma\leq 4$, or $\beta=3$ and $\gamma=1$. Thus we consider the following three subcases.

\noindent{\bf \underline{Subcase 1: $\beta=3$ and $\gamma=1$.}} More precisely, we consider the ideal
\begin{align*}
I=(x^{a}, y^5 - x^{4} z, z^{a-2}, x^{a-4}y^2, y^2z^{a-3}).
\end{align*}
Let $L$ be a general linear form. Let us show that
$$\times L: [R/I]_{a-2}\longrightarrow [R/I]_{a-1}$$
is an isomorphism. 
Set $\point=(x^{a}, y^5 - x^{4} z, z^{a-2}, x^{a-4}y^2)$ and $\aa =(x^{a}, y^5 - x^{4} z, z^{a-2})$. We go along the same lines as  in the proof of Subcase~2 of Case 1  to show that 
\begin{align*}
\times L: [R/\point]_{a-2}\longrightarrow [R/\point]_{a-1}\quad \text{and}\;\times L: [R/I]_{a-2}\longrightarrow [R/I]_{a-1}
\end{align*}
are injective.

\noindent{\bf \underline{Subcase 2: $\beta=1$ and $\gamma=1$.}}
Let $L$ be a general linear form. We will prove that
$$\times L: [R/I]_{a-1}\longrightarrow [R/I]_{a}$$
is an isomorphism. Set $\point=(x^{a}, y^5-x^{4}z, z^{a-2}, x^{a-4}y^{4})$ and $\aa =(x^{a}, y^5-x^{4}z, z^{a-2})$. Then $I_{a-1}=\point_{a-1},\ I_{a}=\point_{a}$. By using the same method as in the proof of Subcase 1 of Case 1, we get $\times L: [R/I]_{a-1} \longrightarrow [R/I]_{a}$
is an isomorphism.

\noindent{\bf \underline{Subcase 3: $\beta=1$ and $2\leq \gamma\leq 4$.}} 
Set $L=x-y-z$. By (ii) and (iii) of Remark~\ref{Remark3.3}, it suffices to show that $y^{a-i}z^i\in J$ for all $0\leq i\leq a-3$, where $$J=((y+z)^{a}, y^5-(y+z)^{5-\gamma} z^\gamma, z^{a-2}, (y+z)^{a-5+\gamma}y^{4}, y^{4}z^{a-2-\gamma}).$$
Clearly, $y^{4}z^{a-4}=y^{4}z^{a-2-\gamma}z^{\gamma-2}\in J$. Now, by using 
$$y^5-(y+z)^{5-\gamma} z^\gamma, y^{4}z^{a-4},(y+z)^{a}\in J,$$
the matrix method in Remark~\ref{Remark3.5} gives a system of $(a-2)$ equations in $(a-2)$ indeterminates $y^{a}, \ldots, y^{3}z^{a-3}$ as follows 
$$MX=U,$$
where $X^t=\bmt{y^{a} & y^{a-1}z &\cdots & y^{3}z^{a-3}}$ and $U^t=\bmt{u_1&u_2&\cdots&u_{a-2}}$ with $u_i\in J$ and $M\in \MS_{a-2}$. It follows that $\det(M)>0$ by Lemma~\ref{Lemmadeterminant}. Thus, by Cramer's rule,  $y^{a-i}z^i\in J$ for all $0\leq i\leq a-3$.

\noindent{\bf \underline{Case 4: $a=c+3$.}} It follows that $R/I$ has the WLP for any $1\leq \beta\leq 3$ or $\beta=4$ and $\gamma\geq 2$ by Theorem~\ref{Theorem3.7}. It suffices to show that $R/I$ has the WLP for $\beta=4$ and $\gamma=1.$
More precisely, we consider the ideal
\begin{align*}
I=(x^{c+3}, y^5 - x^{4} z, z^c, x^{c-1}y, yz^{c-1}).
\end{align*}
Let $L$ be a general linear form. Set $\point=(x^{c+3}, y^5 - x^{4} z, z^c,yz^{c-1})$ and $\aa=(x^{c+3}, y^5 - x^{4} z, z^c)$. By using the same method as in the proof of Subcase 2 of Case 1, we get  $\times L: [R/I]_{c} \longrightarrow [R/I]_{c+1}$ is injective. 
\end{proof}

Now we have actually proved the following result.
\begin{Theo}  \label{Theorem5.3}
Let  $I$ be as in Setting~\ref{IdealGorenstein}. If one of $a,b$ and $c$ is equal to five, then $R/I$ has the WLP.
\end{Theo}
\begin{proof}
It is assumed that $a\geq c$. Therefore, it suffices to consider the cases where $b=5$ or $c=5$. The theorem has been proved in Propositions~\ref{Proposition5.1} and~\ref{Proposition5.2}.
\end{proof}

Finally, after the revised version of this  paper was completed, Lisa Nicklasson kindly informed us that O. Gasanova, S. Lundqvist and L. Nicklasson in \cite{Lisa2020} also proved the following result.
\begin{Theo}  
	Let  $I$ be as in Setting~\ref{IdealGorenstein}. Then $R/I$ has the strong Lefschetz property.  Namely, there is a linear form $\ell\in [R/I]_1$ such that the multiplication
	$$\times \ell^s: [R/I]_{i}\longrightarrow [R/I]_{i+s}$$
	has maximal rank for all $i,s$. In particular, this algebra has the WLP.
\end{Theo}
\section*{Acknowledgments}
The authors are grateful to the referee for very carefully reading the manuscript and  for the many helpful comments that improved the presentation of the article. 
The first author was partially supported by the grant MTM2016-78623-P. The second author was partially supported by the  grant 2017SGR00585.
Part of this work was done while the second author was visiting the Vietnam Institute for Advance Study in Mathematics (VIASM), he would like to thank the VIASM for hospitality and financial support. Finally, the second author would like to thank Hue University for the support.

\bibliographystyle{plain} 
\bibliography{WLP_Gorenstein3} 

\end{document}